\documentclass{amsart}
\usepackage{graphicx}
\usepackage{amsmath}
\usepackage{amsthm}
\usepackage{amsfonts}
\usepackage{amssymb}
\usepackage{euscript}
\usepackage[all,tips]{xy}
\usepackage{epic,eepic}
\usepackage{color}
\newtheorem{theorem}{Theorem}[section]

\newtheorem{proposition}[theorem]{Proposition}
\newtheorem{example}[theorem]{Example}
\theoremstyle{definition}
\newtheorem{definition}[theorem]{Definition}
\theoremstyle{remark}
\newtheorem{remark}[theorem]{Remark}
\numberwithin{equation}{section}

\newcommand{\K}{\mathbb K}

\newcommand{\tone}{{
\begin{picture}(2,10)(0,3)
\drawline(1,0)(1,10)
\put(1,10){\circle*{2}}
\end{picture}
}}

\newcommand{\ttwo}{{
\begin{picture}(16,14)(0,5)
\drawline(8,0)(8,6)(0,14)
\drawline(8,6)(16,14)
\put(0,14){\circle*{2}}
\put(16,14){\circle*{2}}
\put(8,6){\circle*{2}}
\end{picture}
}}

\newcommand{\tthreeone}{{
\begin{picture}(16,14)(0,5)
\drawline(8,0)(8,6)(0,14)
\drawline(8,6)(16,14)
\drawline(12,10)(8,14)
\put(0,14){\circle*{2}}
\put(16,14){\circle*{2}}
\put(8,14){\circle*{2}}
\put(8,6){\circle*{2}}
\put(12,10){\circle*{2}}
\end{picture}
}}

\newcommand{\tthreetwo}{{
\begin{picture}(16,14)(0,5)
\drawline(8,0)(8,6)(0,14)
\drawline(4,10)(8,14)
\drawline(8,6)(16,14)
\put(0,14){\circle*{2}}
\put(16,14){\circle*{2}}
\put(8,14){\circle*{2}}
\put(8,6){\circle*{2}}
\put(4,10){\circle*{2}}
\end{picture}
}}

\newcommand{\tfourone}{{
\begin{picture}(18,15)
\drawline(9,-4)(9,2)(0,11)
\drawline(9,2)(18,11)
\drawline(12,5)(6,11)
\drawline(15,8)(12,11)
\put(0,11){\circle*{2}}
\put(18,11){\circle*{2}}
\put(6,11){\circle*{2}}
\put(12,11){\circle*{2}}
\put(9,2){\circle*{2}}
\put(12,5){\circle*{2}}
\put(15,8){\circle*{2}}
\end{picture}
}}

\newcommand{\tfourtwo}{{
\begin{picture}(18,15)
\drawline(9,-4)(9,2)(0,11)
\drawline(9,2)(18,11)
\drawline(12,5)(6,11)
\drawline(9,8)(12,11)
\put(0,11){\circle*{2}}
\put(18,11){\circle*{2}}
\put(6,11){\circle*{2}}
\put(12,11){\circle*{2}}
\put(9,2){\circle*{2}}
\put(12,5){\circle*{2}}
\put(9,8){\circle*{2}}
\end{picture}
}}

\newcommand{\tfourthree}{{
\begin{picture}(18,15)
\drawline(9,-4)(9,2)(0,11)
\drawline(3,8)(6,11)
\drawline(9,2)(18,11)
\drawline(15,8)(12,11)
\put(0,11){\circle*{2}}
\put(18,11){\circle*{2}}
\put(6,11){\circle*{2}}
\put(12,11){\circle*{2}}
\put(9,2){\circle*{2}}
\put(3,8){\circle*{2}}
\put(15,8){\circle*{2}}
\end{picture}
}}

\newcommand{\tfourfour}{{
\begin{picture}(18,15)
\drawline(9,-4)(9,2)(0,11)
\drawline(6,5)(12,11)
\drawline(9,8)(6,11)
\drawline(9,2)(18,11)
\put(0,11){\circle*{2}}
\put(18,11){\circle*{2}}
\put(6,11){\circle*{2}}
\put(12,11){\circle*{2}}
\put(9,2){\circle*{2}}
\put(6,5){\circle*{2}}
\put(9,8){\circle*{2}}
\end{picture}
}}

\newcommand{\tfourfive}{{
\begin{picture}(18,15)
\drawline(9,-4)(9,2)(0,11)
\drawline(3,8)(6,11)
\drawline(6,5)(12,11)
\drawline(9,2)(18,11)
\put(0,11){\circle*{2}}
\put(18,11){\circle*{2}}
\put(6,11){\circle*{2}}
\put(12,11){\circle*{2}}
\put(9,2){\circle*{2}}
\put(6,5){\circle*{2}}
\put(3,8){\circle*{2}}
\end{picture}
}}

\begin{document}
\title[Ahmed Zahari\textsuperscript{2} and Sania Asif \textsuperscript{1,*$\dag$}]{Derivations, Cohomology  and Deformation of BiHom-Associative Dialgebra}
	\author{ Ahmed Zahari\textsuperscript{2}, Sania Asif \textsuperscript{1,*$\dag$}}
\address{\textsuperscript{1$\dag$}School of Mathematics and Statistics, Nanjing University of information science and Technology, 210044, Nanjing, Jiangsu Province, PR China.}
\address{\textsuperscript{2}Universit\'{e} de Haute Alsace, IRIMAS-D\'{e}partement de Math\'{e}matiques, 18, rue des Fr\`eres Lumi\`ere F-68093 Mulhouse, France.}
\email{\textsuperscript{2}zaharymaths@gmail.com}
	\email{\textsuperscript{1,*$\dag$}11835037@zju.edu.cn}
%
\begin{abstract} Due to the immense importance of BiHom Type algebras and cohomology of various algebraic structures, this paper is devoted to defining the BiHom-associative dialgebra, its derivation, generalized derivation, and quasi-derivation. We  provided the complete classification of these derivations of $2-$ and $3$-dimensional BiHom-associative dialgebras. We further generalized the cohomology of BiHom-associative algebras to the cohomology of  BiHom-associative dialgebras. As an application to cohomology, we evaluate the one-parameter formal deformation of BiHom-associative dialgebras.
\end{abstract}
\footnote{Both authors contributed equally. \\
Corresponding Emails:{\textsuperscript{1,*$\dag$}11835037@zju.edu.cn};{\textsuperscript{2}zaharymaths@gmail.com}}
\maketitle
\section{introduction}In this paper, we deal with certain types of algebras, called BiHom-type algebras. In these algebras, the identities defining the structures are twisted by two homomorphisms. Recently, BiHom-type algebras have been studied by many authors. The notion of BiHom-Lie algebras was first introduced by a G.Grazini et. al. in \cite{GACF}. BiHom-Lie algebras appeared as an extension of Hom-Lie algebra (Lie algebra structure twisted by a single homomorphism map). Later on, BiHom algebras are studied with reference to BiHom-Lie algebra, BiHom-associative algebra, BiHom-Leibniz algebra, and BiHom-Transposed algebra in \cite{KA, ML}. BiHom-bialgebras structure equipped with BiHom-coalgebra and BiHom-algebra structure under some compatible conditions is studied in \cite{LAF, LCP}. In addition, the BiHom structure of various operator algebras and operads is also studied in \cite{XL}.\\
 In view of the huge importance of BiHom algebra structures, our interest in this paper is to study BiHom-associative dialgebra or BiHom-diassociative algebra, see \cite{MM, P, RRB, RRB1} for more details. The associative dialgebras (also known as diassociative algebras) as  a generalization of associative algebras has been introduced by Loday in 1990 (see \cite{L} and references therein). Originally, an associative dialgebra  is a triple $(D, \dashv, \vdash)$  consisting of a vector space $D$, and two multiplication maps $\vdash, \dashv$ that satisfies the following axioms: \begin{eqnarray}
(x\dashv y)\dashv z&=& x\dashv(y\dashv z), \\
(x\dashv y)\dashv z&=& x\dashv(y\vdash z),\\
(x\vdash y)\dashv z&=& x\vdash(y\dashv z), \\
(x\dashv y)\vdash z&=& x\vdash(y\vdash z), \\
(x\vdash y)\vdash z&=& x\vdash(y\vdash z), 
\end{eqnarray}for all $x, y, z\in D$.
They are a generalization of associative algebras in the sense that they possess two associative multiplications. When the two associative axioms are equal we recover associative algebra. The study of these algebras is important because of their wide range of applications in various disciplines of mathematics including computer science and  physics. In pure mathematics, its motivation lies in  classical geometry and non-commutative geometry. Moreover, a better understanding of these structures can lead to  the development of better models for physical phenomena. In particular, disassociative algebra involves its motivation to devise an algebra whose commutator leads to a Leibniz algebra, similar to the relationship between Lie algebra and associative algebra. Another motivation comes from the research of an obstruction to the periodicity in algebraic $ K$ theory. Just like other algebraic structures, when twisted by one or two morphism maps, we get Hom and BiHom algebras. One can also get BiHom associative dialgebra, by twisting the disassociative algebra structure. \par
Cohomology of algebra is also an important mathematical structure, which tells us about the number of ways a thing can go wrong. It is always been a center of attraction for many researchers and significant research in this direction can be seen in \cite{YH, AD1}. Cohomology of various BiHom algebras structures  is developed  \cite{AM}. Loday has constructed a cohomology theory for associative  dialgebras, that involves the use of planar binary trees in the construction of the dialgebra complex that defines the cohomology. Dialgebras cohomology with the coefficients was developed in \cite{AF}. Later on Das in \cite{AD} studied the cohomology of BiHom-associative algebra inherits a Gerstenhaber  structure and controls formal deformations. The same relation also holds for associative dialgebras that can be seen in \cite{L, MM}. Taking motivation from the above-cited papers on BiHom type algebras, cohomology, and associative dialgebras, we managed to combine all these structures into a single algebraic structure called "BiHom-associative dialgebra" and evaluated its cohomology and formal deformation.\par
This paper is organized as follows: In Section $2$, we defined the notion of BiHom-associative dialgebra, and showed that for the identical twist maps, our theory of BiHom-associative dialgebra coincides with associative algebras and  Hom-associative dialgberas. In Section $3$, we study the derivations, generalized derivations, and quasi-derivation of BiHom associative dilagebras and presented many important propositions and examples for a better understanding of the algebraic structure under consideration. We further provided the classification of derivations, quasi derivations, and generalized derivations of $2-$ and $3$-dimensional BhiHom- associative dialgebras. In Section $4$, we first recall the notion of the cohomology of BiHom-associative algebra and then generalize it to the cohomology of  BiHom-associative dialgebras. As an application to the cohomology, we evaluated one-parameter foraml deformation of BiHom-associative dialgebras.
\section{Preliminaries}In this section we give some important definitions and useful remarks to support our findings in the next sections.
\begin{definition}\label{dia}
A BiHom-associative dialgebras is a $5$-tuple $(A, \dashv, \vdash, \varphi, \psi)$ consisting of a  linear space $A$ linear maps
 $\dashv, \vdash,: A\times A \longrightarrow A$ and  $\varphi, \psi : A\longrightarrow A$, satisfying the following
 conditions : 
\begin{eqnarray}
\varphi\circ\psi&=&\psi\circ\varphi,\\
(x\dashv y)\dashv\psi(z)&=&\varphi(x)\dashv(y\dashv z),\label{eq1}\\
(x\dashv y)\dashv\psi(z)&=&\varphi(x)\dashv(y\vdash z),\label{eq2}\\
(x\vdash y)\dashv\psi(z)&=&\varphi(x)\vdash(y\dashv z),\label{eq3}\\
(x\dashv y)\vdash\psi(z)&=&\varphi(x)\vdash(y\vdash z),\label{eq4}\\
(x\vdash y)\vdash\psi(z)&=&\varphi(x)\vdash(y\vdash z),\label{eq5}
\end{eqnarray}for all $x, y, z\in A$. 
\end{definition}
We called $\varphi$ and $\psi$ ( in this order ) the structure maps of A.
\begin{remark}
 A BiHom-associative dialgebra in which $\dashv~=~\vdash$ is called a BiHom-associative algebra \cite{AI}.
\end{remark}
\begin{definition}
 Let $({D}, \dashv, \vdash, \varphi, \psi)$ and $({D}', \dashv',\vdash', \varphi', \psi')$ be two  BiHom-associative dialgebras. A linear map 
$f : {D}\rightarrow {D}'$ is a morphism of BiHom-associative dialgebras if
$$\varphi'\circ f=f\circ\varphi,\quad \psi'\circ f=f\circ\psi,\quad f \circ\dashv ~=~\dashv\circ( f\otimes f)\;\;\mbox{and}\;\; f \circ\vdash
 =\vdash'\circ( f\otimes f).$$
\end{definition}
\begin{definition}
A BiHom-associative dialgebra $(A, \dashv, \vdash, \varphi, \psi)$ in which $\varphi$ and $\psi$ are morphisms is said to be a multiplicative
 BiHom-associative dialgebra.\\
If moreover, $\varphi$ and $\psi$ are bijective (i.e. automorphisms), then $(A, \dashv, \vdash, \varphi, \psi)$ is said to be a 
regular BiHom-associative dialgebra.
\end{definition}

\begin{remark}
 A BiHom-associative dialgebra in which $\varphi=\psi$ is said to be a Hom-associative dialgebra \cite{L}.
 If in addition, $\varphi=\psi=Id$, $A$ is called an associative dialgebra. More precisely, A Hom-associative dialgebra is a quadruple 
$(A, \dashv, \vdash, \phi)$ such that
\begin{eqnarray}
(x\dashv y)\dashv\varphi(z)&=&\varphi(x)\dashv(y\dashv z),\label{eq6}\\
(x\dashv y)\dashv\varphi(z)&=&\varphi(x)\dashv(y\vdash z),\label{eq7}\\
(x\vdash y)\dashv\varphi(z)&=&\varphi(x)\vdash(y\dashv z),\label{eq8}\\
(x\dashv y)\vdash\varphi(z)&=&\varphi(x)\vdash(y\vdash z),\label{eq9}\\
(x\vdash y)\vdash\varphi(z)&=&\varphi(x)\vdash(y\vdash z).\label{eq10}
\end{eqnarray}
while an associative dialgebra is a triple $(A, \dashv, \vdash)$ such that
\begin{eqnarray}
(x\dashv y)\dashv z&=&x\dashv(y\dashv z),\label{eq11}\\
(x\dashv y)\dashv z&=&x\dashv(y\vdash z),\label{eq12}\\
(x\vdash y)\dashv z&=&x\vdash(y\dashv z),\label{eq13}\\
(x\dashv y)\vdash z&=&x\vdash(y\vdash z),\label{eq14}\\
(x\vdash y)\vdash z&=&x\vdash(y\vdash z).\label{eq15}
\end{eqnarray}
\end{remark}
\section{The  $(\varphi^k, \psi^l)$-Derivation of BiHom-associative dialgebras}
In this section, we introduce and study derivations of BiHom-associative dialgebras. For any non negative integer $k$ and $l$ by $\varphi^k$ the $k$-times composition of $\varphi$ and $\psi^l$ the $l$-times composition of $\psi$, 
i.e. $\varphi^k=\varphi\circ\cdots\circ\varphi$ ($k$-times), $\psi^l=\psi\circ\cdots\circ\psi$ ($l$-times). 
Since the maps $\varphi, \psi$ commute, we denote by 
$\varphi^k\psi^l=\underbrace{\varphi\circ\cdots\circ\varphi}_{k-\text{times}}\circ\underbrace{\psi\circ\cdots\circ\psi}_{l-\text{times}}.$
 In particular, $\varphi^0\psi^0=id$ and $\varphi^1\psi^1=\varphi\psi$.Moreover, if $({D}, \dashv,\vdash  , \varphi, \psi)$ is a regular 
BiHom-associative dialgebra, $\varphi^{-k}\psi^{-l}$ is the inverse of 
$\varphi^{k}\psi^{l}.$ Before defining the derivation of BiHom associative dialgebra, we first provide the definition of the derivation of BiHom-associative algebra.
\begin{definition}
Let $(A,\mu, \varphi, \psi)$ be a BiHom-associative algebra. A linear map ${D} : A\longrightarrow  A$ 
is called an $(\varphi^k, \psi^l)$-derivation of 
$(A,\mu, \varphi, \psi)$, if it satisfies  
$$
{D}\circ\varphi=\varphi\circ {D}\quad \text{and}\quad{D}\circ\psi=\psi\circ {D} 
$$
$$
{D}\circ\mu(x, y)=\mu({D}(x), \varphi^k \psi^l(y))+\mu(\varphi^k \psi^l(x), {D}(y)).
$$
\end{definition}
\begin{definition}
A linear map $d : {A}\longrightarrow {A}$ on a BiHom-associative algebra $A$ is called a differential if 
$d(x\cdot y)=dx\cdot y+x\cdot dy,\, x,y\in A, \, d^2=0$,  $d\circ \varphi=\varphi\circ d$ and $d\circ\psi=\psi\circ d.$
\end{definition} The following proposition provides when a differential BiHom-associative algebra can turn into BiHom-associative dialgebra.
\begin{proposition}
Let $(A, \cdot, \varphi, \psi, d)$ be a differential BiHom-associative algebra (i.e. a BiHom-associative algebra equipped with a differential).
 Consider the products $\dashv$ and $\vdash$ on A given by
 $$x\dashv y=\varphi(x)\cdot dy\quad\mbox{and}\quad x\vdash y=dx\cdot\psi(y).$$
 Then  $(A, \dashv,\vdash, \varphi, \psi)$ is a  BiHom-associative dialgebra, provided that $\varphi$ and $\psi$ idempotant structural maps.
\end{proposition}
\begin{proof}
It is straightforward  to check the BiHom-associativity. 
$$\begin{array}{ll}
\varphi(x)\dashv(y\dashv z)&=\varphi(x)\dashv(\varphi(y)\cdot d(z))\\
&=\varphi\circ\varphi(x)\cdot d(\varphi(y)\cdot d(z))\\
&=(\varphi(x)\cdot d(\varphi(y))\cdot d\psi(z)\\&=(x\dashv y)\dashv \psi(z),
\end{array}$$
for any $x, y, z\in A$.
\end{proof}
\begin{proposition}
Let $({D}, \dashv,\vdash  , \varphi, \psi)$ be a BiHom-associative dialgebra with $\varphi(a)=\psi(a)$, where $a\in A$. Let us define, for any $x,y\in A$,  the linear map 
 $Ad_a^{(\varphi, \psi)} : A\longrightarrow A$ by 
 $$Ad_a^{(\varphi, \psi)}(x)=x\dashv\psi(a)-\varphi(a)\vdash x.$$
Then $Ad_a^{(\varphi, \psi)}$ is a derivation of $A$.
\end{proposition}
\begin{proof}
We have 
$$\begin{array}{ll}
Ad_a^{(\varphi, \psi)}(x)\dashv y+x\dashv Ad_a^{(\varphi, \psi)}(y)
&=(x\dashv\psi(a)-\varphi(a)\vdash x)\dashv y+x\dashv(y\dashv\psi(a)-\varphi(a)\vdash y)\\
&=(x\dashv\psi(a))\dashv y-(\varphi(a)\vdash x)\dashv y+x\dashv(y\dashv\psi(a))-x\dashv(\varphi(a)\vdash y)\\
&=(x\dashv y)\dashv\psi(a)-\varphi(a)\vdash (x\dashv y)\\&=Ad_a^{(\varphi, \psi)}(x\dashv y).
\end{array}$$
So, $Ad_a^{(\varphi, \psi)}$ is a derivation.
\end{proof}
 Like the BiHom-associative algebra case, we define $(\varphi^k, \psi^l)$-derivation for a BiHom-associative dialgebra as follows.
\begin{definition}
Let $({A}, \dashv, \vdash, \varphi, \psi)$ be a BiHom-associative dialgebra. A linear map ${D} : {A}\rightarrow {A}$ is called a 
$(\varphi^k, \psi^l)$-derivation of ${A}$ if it satisfies 
\begin{eqnarray}
 {D}\circ\varphi=\varphi\circ{D},\,{D}\circ\psi=\psi\circ{D},\\
	{D}(x\dashv y)=\varphi^k\psi^l(x)\dashv{D}(y)+{D}(x)\dashv\varphi^k\psi^l(y),\\
 {D}(x\vdash y)=\varphi^k\psi^l(x)\vdash {D}(y)+{D}(x)\vdash\varphi^k\psi^l(y),
\end{eqnarray}
for $x, y\in   {A}.$ The set of all $(\varphi^k,\psi^l)$-derivations of ${A}$ is denoted by $Der_{(\varphi^k, \psi^l)}({A})$  Where the space of all derivation on BiHom associative dialgebra is denoted by $Der({A}):=\displaystyle\bigoplus_{k\geq 0}\displaystyle\bigoplus_{l\geq 0}Der_{(\varphi^k, \psi^l)}({A})$.
\end{definition} 
\begin{example}
Consider the $3$-dimensional BiHom-associative dialgebra given in  \cite{AI} with basis 
$\left\{e_1, e_2, e_3\right\}$, where the multiplication and structure maps are defined as \begin{eqnarray*}
   & e_2\dashv e_2=ae_1,~~
    e_2\dashv e_3=be_1,~~
    e_3\dashv e_2=ce_1,\\&
     e_2\vdash e_1=e_1,~~
     e_2\vdash e_2=de_1,~~
     e_3\vdash e_2=fe_1,\\& 
      e_1\dashv e_2=e_2\dashv e_1=e_1,
\end{eqnarray*}
and  
$$\varphi(e_2)=e_1,\, \psi(e_2)=e_1,\, \psi(e_3)=be_3.$$
 A direct computation gives the following derivation: 
\begin{eqnarray}
 {D}_{\varphi\psi}(e_1)&=&-\frac{(af-cd)d_{12}}{f}e_2-\frac{(af-cd)d_{12}}{f}e_3,\nonumber\\
 {D}_{\varphi\psi}(e_2)&=&d_{22}e_2+d_{23}e_3,\nonumber\\
{D}_{\varphi\psi}(e_3)&=&-\frac{dd_{22}}{f}e_2-\frac{dd_{23}}{f}e_3.\nonumber
\end{eqnarray}
\end{example}
\begin{proposition}\label{p7}
For any ${D}\in Der_{(\varphi^k, \psi^l)}(A)$ and ${D}'\in Der_{(\varphi^{k'},\psi^{l'})}(A)$, we have 
$$\left[{D},{D'}\right]\in Der_{(\varphi^{k+k'}, \psi^{l+l'})}(A).$$
\end{proposition}
\begin{proof}
For $x, y\in A$, we have 
$$\begin{array}{ll}
\left[{D}, {D'}\right]\circ\mu(x, y)
&={D}\circ{D'}\circ\mu(x, y)-{D'}\circ{D}\circ\mu(x, y)\\
&={D}(\mu({D}'(x), \varphi^k\psi^l(y))+\mu(\varphi^k\psi^l(x),{D}'(y)))\\
&\quad-{D}'(\mu({D}(x), \varphi^k\psi^l(y))+\mu(\varphi^k\psi^l(x),{D}(y)))\\
&=\mu({D}\circ{D'}(x),\varphi^{k+k'}\psi^{l+l'}(y))+\mu(\varphi^k\psi^l\circ{D'}(x),{D}\circ\varphi^k\psi^l(y))\\
&\quad+\mu({D}\circ\varphi^k\psi^l(x),\varphi^k\psi^l\circ{D}'(y))+\mu(\varphi^{k+k'}\psi^{l+l'}(x),{D}\circ{D'}(y))\\
&\quad-\mu({D'}\circ{D}(x),\varphi^{k+k'}\psi^{l+l'}(y))-\mu(\varphi^k\psi^l\circ{D}(x),{D'}\circ\varphi^k\psi^l(y))\\
&\quad-\mu({D'}\circ\varphi^k\psi^l(x),\varphi^k\psi^l{D}(y))-\mu(\varphi^{k+k'}\psi^{l+l'}(x),{D'}\circ{D}(y)).
\end{array}$$
Since ${D}$ and ${D}'$ satisfy 
${D}\circ\varphi=\varphi\circ{D},\,{D}'\circ\varphi=\varphi\circ{D}'$,\,  
${D}\circ\psi=\psi\circ{D},\,{D}'\circ\psi=\psi\circ{D}'$.\\ We obtain  
$\varphi^k\psi^l\circ{D'}={D'}\circ\varphi^k\psi^l,\, {D}\circ\varphi^{k'}\psi^{l'}=\varphi^{k'}\psi^{l'}\circ{D}.$
Therefore, we get\\
$$\left[{D},{D'}\right]\circ\mu(x, y)=\mu(\varphi^{k+k'}\psi^{l+l'}(x),\left[{D},{D'}\right](y))+
\mu(\left[{D},{D'}\right](x),\varphi^{k+k'}\psi^{l+l'}(y)).$$
Furthermore, it is straightforward to see that 
$$\begin{array}{ll}
\left[{D},{D'}\right]\circ\varphi
&={D}\circ{D'}\circ\varphi-{D'}\circ{D}\circ\varphi\\
&=\varphi\circ{D}\circ{D}'-\varphi\circ{D}'\circ{D}=\varphi\circ\left[{D},{D'}\right].
\end{array}$$ 
$$\begin{array}{ll}
\left[{D},{D'}\right]\circ\psi
&={D}\circ{D'}\circ\psi-{D'}\circ{D}\circ\psi\\
&=\psi\circ{D}\circ{D}'-\psi\circ{D}'\circ{D}=\psi\circ\left[{D},{D'}\right]
\end{array}$$
which yields that $\left[{D},{D'}\right]\in Der_{(\varphi^{k+k'}, \psi^{l+l'})}(A)$ by considering $\mu=\dashv$ and $\mu=\vdash$ respectively.
\end{proof}
 
\begin{proposition}
The space $Der_{(\varphi^{k}, \psi^{l})}(A)$ is an invariant of the triple BiHom-associative dialgebra A.
\end{proposition}
\begin{proof}
Let $\sigma :  (A, \dashv_A, \vdash_A,  \varphi^k, \psi^l)\longrightarrow (B, \dashv_B, \vdash_B, \varphi^k, \psi^l)$  be a triple BiHom-associative dialgebra isomorphism
and let ${D}$ be a $(\varphi^k, \psi^l)$-derivation of A. Then for any $x, y, z\in B$. We have :  
$$\begin{array}{ll}
\sigma{D}\sigma^{-1}\circ(((x)\dashv_B(y))\dashv_B (z))
&=\sigma{D}\circ((\sigma^{-1}(x)\dashv_A\sigma^{-1}(y))\dashv_A\sigma^{-1}(z))\\
&=\sigma({D}\circ\sigma^{-1}(x)\vdash_A\sigma^{-1}\circ\varphi^s\psi^r(y))\vdash_A\sigma^{-1}\circ\varphi^k\psi^l(z))\\
&+\sigma(\sigma^{-1}\circ\varphi^k\psi^l(x)\vdash_A{D}\circ\sigma^{-1}(y))\vdash_A\sigma^{-1}\circ\varphi^k\psi^l(z))\\
&+\sigma(\sigma^{-1}\circ\varphi^k\psi^l(x)\vdash_A\sigma^{-1}\circ\varphi^k\psi^l(y)\vdash_A{D}\circ\sigma^{-1}(z))\\
&=({D}\circ\sigma^{-1}(x)\dashv_B\varphi^k\psi^l(y))\dashv_B\varphi^k\psi^l(z))\\
&+(\varphi^k\psi^l(x)\dashv_B\sigma\circ{D}\circ\sigma^{-1}(y))\dashv_B\varphi^k\psi^l(z))\\
&+(\varphi^k\psi^l(x)\dashv_B\varphi^k\psi^l(y))\dashv_B{D}\circ\sigma^{-1}(z)).
\end{array}$$
Thus $\sigma\circ{D}\circ\sigma^{-1}$ is a $(\varphi^k, \psi^l)$-derivation of $B$. Hence the mapping
$\psi : Der_{(\alpha^{k},\beta^{l})}(A)\longrightarrow Der_{(\varphi^k, \psi^{l})}(B)$, defined by ${D}\longmapsto \sigma{D}\sigma^{-1}$
is an isomorphism of triple  BiHom-associative dialgebras. 

In fact, it is easy to see that $\psi$ is linear. Moreover, let 
${D}_1, {D}_2, {D}_3$ be derivations of $A$ : 
$$\begin{array}{ll}
&\varphi^k\psi^l\circ\Psi({D}_1\dashv_{tr}{D}_2)\dashv_{tr}{D}_3)\\
&=\varphi^k\psi^l\Psi(tr({D}_1)({D}_2\dashv{D}_3))+\varphi^k\psi^l\Psi(tr({D}_3)({D}_1{D}_2))
+\varphi^k\psi^l\Psi(tr({D}_2)({D}_3\dashv {D}_1))\\
&=\varphi^k\psi^l tr({D}_1)\Psi({D}_2\dashv{D}_3)
+\varphi^k\psi^l tr({D}_3)\Psi({D}_1\dashv{D}_2)+\varphi^k\psi^l tr({D}_2)\Psi({D}_3\dashv{D}_1)\\
&=\varphi^k\psi^l tr(\Psi({D}_1))(\Psi({D}_2)\dashv\Psi({D}_3))+\varphi^k\psi^l tr(\Psi({D}_3))(\Psi({D}_1)\dashv\psi({D}_2))\\
&+\varphi^k\psi^l tr(\Psi({D}_2))\Psi((\Psi({D}_3)\dashv\Psi({D}_1)),
\end{array}$$
since $\psi$ is a morphism of the $Der_{(\varphi^k, \psi^l)}(A)$ and $Der_{(\varphi^k, \psi^l)}(B)$, and 
$tr({D})=tr(\sigma\circ{D}\circ\sigma^{-1}).$\\
Then $\varphi^k\psi^l\Psi(({D}_1\dashv_{tr}{D}_2)\dashv_{tr}{D}_3))=\varphi^k\psi^l((\Psi({D}_1)\dashv_{tr}\Psi({D}_2))\dashv_{tr}\Psi({D}_3)).$
\end{proof}

\begin{definition}
Let $({A}, \dashv, \vdash, \varphi, \psi)$ be a BiHom-associative dialgebra and $\alpha, \beta, \gamma$ elements of $\mathbb{C}$. 
A linear map ${D\in End(A)}$ is a generalized 
$(\varphi^k, \psi^l)$-derivation or a $(\alpha, \beta, \gamma)$-$(\varphi^k, \psi^l)$-derivation of ${A}$ if for all $x, y\in {A}$, we have  

\begin{eqnarray}\label{rd0}
\varphi\circ D=D\circ\varphi&,& \psi\circ D=D\circ\psi\\
\alpha({D}(x\dashv y))&=&\beta(\varphi^k\psi^l(x)\dashv{D}(y))+\gamma({D}(x)\dashv\varphi^k\psi^l(y))\\
\alpha({D}(x\vdash y))&=&\beta(\varphi^k\psi^l(x)\vdash {D}(y))+\gamma({D}(x)\vdash\varphi^k\psi^l(y)).
\end{eqnarray}
We denote the set of all $(\alpha, \beta, \gamma)$-$(\varphi^k, \psi^l)$-derivation by $Der_{(\varphi^k, \psi^l)}^{(\alpha, \beta, \gamma)}({A})$. 
 Moreover, we have the following space 
 $$Der^{(\alpha, \beta, \gamma)}({A})=\left\{D\in Der_{(\varphi^k, \psi^l)}^{(\alpha, \beta, \gamma)}({A})|k, 
l\in \mathbb{Z}\right\}.$$
\end{definition}

\begin{proposition}
For any $D\in Der_{(\varphi^k, \psi^l)}^{(\alpha, \beta, \gamma)}({A})$ and $D'\in {Der}_{(\varphi^{k'}, \psi^{l'})}^{(\alpha^{'}, \beta^{'}, \gamma^{'})}({A})$, we can define the commutator $[D,D']={D}\circ{D'}-{D'}\circ{D}$, such that $[D,D']\in {Der}_{(\varphi^{k+k'}, \psi^{l+l'})}^{(\alpha \alpha ^{'}, \beta \beta ^{'}, \gamma \gamma ^{'})}({A})$ for all $\alpha,\beta , \gamma, \alpha',\beta' , \gamma'\in \mathbb{C}$ where $k,k',l,l'\in \mathbb{Z}$.
\end{proposition}
\begin{proof}
The proof is similar to Proposition \ref{p7}.
\end{proof}
Now consider the $2-$ and $3-$ dimensional BiHom-associative dialgebras structures taken from our previous paper \cite{AI}.
\begin{center}
 \begin{tabular}{||c||c||c||c||c||c||c||c||c||c||c||c||}
\hline
$Algebras$&Multiplications &Morphisms $\varphi,\psi$.
\\ \hline 
$\mathcal{A}lg_1$&
$\begin{array}{ll}  
e_1\dashv e_2=ae_1,
e_2\dashv e_1=be_1,\\
e_1\vdash e_2=ce_1,
e_2\vdash e_1=de_1,
e_2\vdash e_2=fe_1,
\end{array}$
&
$\begin{array}{ll}  
\varphi(e_2)=e_1,\\
\psi(e_2)=e_1.
\end{array}$
\\ \hline 
$\mathcal{A}lg_2$
&
$\begin{array}{ll}  
e_1\dashv e_2=ae_1,
e_2\dashv e_1=ae_1,\\
e_2\dashv e_2=e_1,
e_1\vdash e_2=e_1,
e_2\vdash e_1=e_1,
\end{array}$
&
$\begin{array}{ll}  
\varphi(e_2)=e_1,\\
\psi(e_2)=e_1.
\end{array}$
\\ \hline
$\mathcal{A}lg_3$
&
$\begin{array}{ll}   
e_1\dashv e_2=ae_1,
e_1\vdash e_2=be_1,\\
e_2\vdash e_1=ce_1,
e_2\vdash e_2=de_1,
\end{array}$
&
$\begin{array}{ll}  
\varphi(e_2)=e_1,\\
\psi(e_2)=e_1.
\end{array}$
\\ \hline  
$\mathcal{A}lg_4$
&
$\begin{array}{ll}   
e_1\dashv e_2=e_1,
e_2\dashv e_1=e_1,
e_2\dashv e_2=ae_1,\\
e_1\vdash e_2=be_1,
e_2\vdash e_1=ce_1,
e_2\vdash e_2=de_1,
\end{array}$
&
$\begin{array}{ll}  
\varphi(e_2)=e_1,\\
\psi(e_2)=e_1.
\end{array}$
\\ \hline  
\end{tabular}
\end{center}
\begin{center}
   \begin{tabular}{||c||c||c||c||c||c||c||c||c||c||c||c||}
\hline
$Algebras$&Multiplications&Morphisms $\varphi,\psi$
\\ \hline 
$\mathcal{A}lg_1$&
$\begin{aligned}
&e_1\dashv e_2=e_1,&\\
&e_2\dashv e_1=e_1,&\\
&e_2\dashv e_2=ae_1,&
\end{aligned}
\begin{aligned}
&e_2\dashv e_3=be_1,&\\
&e_3\dashv e_2=ce_1,&\\
&e_2\vdash e_1=e_1,&
\end{aligned}
\begin{aligned}
&e_2\vdash e_2=de_1,&\\
&e_3\vdash e_2=fe_1,&
\end{aligned}$
&
$\begin{array}{ll}  
\varphi(e_2)=e_1,\\  
\psi(e_2)=e_1,\\
\psi(e_3)=be_3.
\end{array}$
\\ \hline 
$\mathcal{A}lg_2$&
$\begin{aligned}
&e_1\dashv e_2=e_1,&\\
&e_2\dashv e_1=e_1,&\\
&e_2\dashv e_2=e_1,&
\end{aligned}
\begin{aligned}
&e_2\dashv e_3=e_1,&\\
&e_3\dashv e_2=e_1,&\\
&e_1\vdash e_2=e_1,&
\end{aligned}
\begin{aligned}
&e_2\vdash e_1=e_1,&\\
&e_2\vdash e_2=e_1,&\\
&e_3\vdash e_2=e_1,&\\
\end{aligned}$
&
$\begin{array}{ll}  
\varphi(e_2)=e_1,\\  
\psi(e_2)=e_1,\\
\psi(e_3)=be_3.
\end{array}$
\\ \hline  
$\mathcal{A}lg_3$&
$\begin{aligned}
&e_1\dashv e_2=e_1,&\\
&e_2\dashv e_1=e_1,&\\
&e_2\dashv e_2=e_1,&
\end{aligned}
\begin{aligned}
&e_2\dashv e_3=e_1,&\\
&e_3\dashv e_2=e_1,&\\
&e_1\vdash e_2=e_1,&
\end{aligned}
\begin{aligned}
&e_2\vdash e_2=e_1,&\\
&e_2\vdash e_3=e_1,&\\
&e_3\vdash e_2=e_1,&\\
\end{aligned}$
&
$\begin{array}{ll}  
\varphi(e_2)=e_1,\\  
\psi(e_2)=e_1,\\
\psi(e_3)=be_3.
\end{array}$
 \\ \hline
$\mathcal{A}lg_4$&
$\begin{aligned}
&e_1\dashv e_2=e_1,&\\
&e_2\dashv e_1=e_1,&\\
&e_2\dashv e_2=e_1,&
\end{aligned}
\begin{aligned}
&e_2\dashv e_3=e_1,&\\
&e_1\vdash e_2=e_1,&\\
&e_2\vdash e_1=e_1,&
\end{aligned}
\begin{aligned}
&e_2\vdash e_2=e_1,&\\
&e_2\vdash e_3=e_1,&\\
&e_3\vdash e_2=e_1,&\\
\end{aligned}$
&
$\begin{array}{ll}  
\varphi(e_2)=e_1,\\  
\psi(e_2)=e_1,\\
\psi(e_3)=be_3.
\end{array}$
\\ \hline 
$\mathcal{A}lg_5$&
$\begin{aligned}
&e_1\dashv e_2=e_1,&\\
&e_2\dashv e_1=e_1,&\\
&e_2\dashv e_2=e_1,&
\end{aligned}
\begin{aligned}
&e_2\dashv e_3=e_1,&\\
&e_1\vdash e_2=e_1,&\\
&e_2\vdash e_1=e_1,&
\end{aligned}
\begin{aligned}
&e_2\vdash e_3=e_1,&\\
&e_3\vdash e_2=e_1,&\\
\end{aligned}$
&
$\begin{array}{ll}  
\varphi(e_2)=e_1,\\  
\psi(e_2)=e_1,\\
\psi(e_3)=be_3.
\end{array}$
\\ \hline
\end{tabular}
\end{center} Using the above tables and derivation structure of BiHom associative dialgebra, we evaluate the classification of derivations on $2-$  and $3-$ dimensional BiHom-associative dialgebras in the next subsections. We use Mathematicae software with Computer Algebra programs to perform our calculations.
\subsection{The Classification of Derivations of BiHom-associative dialgebras}

\begin{itemize}
    \item The classification table for $2$-dimensional derivation is given as follows:
\begin{center} \begin{tabular}{||c||c||c||c||c||c||c||c||c||c||c||c||}
\hline
IC&$\textbf{D}_{\varphi\psi}$&Dim&IC&$\textbf{D}_{\varphi\psi}$&Dim\\
 \hline
$\mathcal{A}lg_2^1$&
$\left(\begin{array}{cccc}
0&0\\
0&d_{12}
\end{array}
\right)$
&
2
&
$\mathcal{A}lg_2^2$
&
$\left(\begin{array}{cccc}
0&0\\
0&d_{12}
\end{array}
\right)$
&
1
\\ \hline 
$\mathcal{A}lg_2^3$&
$\left(\begin{array}{cccc}
0&0\\
0&d_{12}
\end{array}
\right)$
&
2
&
$\mathcal{A}lg_2^4$
&
$\left(\begin{array}{cccc}
0&0\\
0&d_{12}
\end{array}
\right)$
&
1
\\ \hline
\end{tabular}
\end{center}
\item The classification table for $3$-dimensional derivation is given as follows:
\begin{center}
    \begin{tabular}{||c||c||c||c||c||c||c||c||c||c||c||c||}
\hline
IC&$\textbf{D}_{\varphi\psi}$&Dim&IC&$\textbf{D}_{\varphi\psi}$&Dim\\
			\hline
$\mathcal{A}lg_3^1$&
$\left(\begin{array}{cccc}
0&d_{12}&0\\
0&0&0\\
0&0&d_{33}
\end{array}
\right)$
&
2
&
$\mathcal{A}lg_3^2$
&
$\left(\begin{array}{cccc}
0&d_{12}&0\\
0&0&0\\
0&0&d_{33}
\end{array}
\right)$
&
2
\\ \hline 
$\mathcal{A}lg_3^3$&
$\left(\begin{array}{cccc}
0&d_{12}&0\\
0&0&0\\
0&0&d_{33}
\end{array}
\right)$
&
2
&
$\mathcal{A}lg_3^4$
&
$\left(\begin{array}{cccc}
0&d_{12}&0\\
0&0&0\\
0&0&d_{33}
\end{array}
\right)$
&
2
\\ \hline
$\mathcal{A}lg_3^5$&
$\left(\begin{array}{cccc}
0&d_{12}&0\\
0&0&0\\
0&0&d_{33}
\end{array}
\right)$
&
2
&

&

&
\\ \hline
\end{tabular}
\end{center}
\end{itemize}

\subsection{The Classification of Quasi-derivation of BiHom-associative dialgebras} In this subsection we define quasi-derivation of BiHom- associative dialgebra and provide the classification for $2$ and $3$ dimensional 
\begin{definition}\label{rd3.9}
Let $({A}, \dashv, \vdash, \varphi, \psi)$ be a BiHom-associative dialgebra. A linear map ${D} : {A}\rightarrow {A}$ is called a 
$(\varphi^k, \psi^l)$-quasi derivation of ${D}$, if there exist a derivation map $D':A\to A$ such that the following equations hold
\begin{enumerate}
	\item [$1.$] ${D}\circ\varphi= \varphi\circ{D},\,{D}\circ\psi=\psi\circ{D}$, ${D'}\circ\varphi= \varphi\circ{D'},\,{D'}\circ\psi=\psi\circ{D'}$,
	\item [$2.$] ${D'}(x\dashv y)=\varphi^k\psi^l(x)\dashv{D}(y)+{D}(x)\dashv\varphi^k\psi^l(y),$
	\item [$3.$] ${D'}(x\vdash y)=\varphi^k\psi^l(x)\vdash {D}(y)+{D}(x)\vdash\varphi^k\psi^l(y),$
\end{enumerate}
for $x, y\in   {A}.$\\
We denote by  $QDer_{(\varphi^k, \psi^l)}({A})$, the set of all $(\varphi^k,\psi^l)$-derivations of ${A}$. The space of all Quasi-derivations of $A$ is given by $QDer({A}):=\displaystyle\bigoplus_{k\geq 0}\displaystyle\bigoplus_{l\geq 0}QDer_{(\varphi^k, \psi^l)}({A})$.
\end{definition}

\begin{itemize}
    \item $2-$dimensional Quasi derivation on BiHom-associative dialgebra: \begin{center}
    \begin{tabular}{||c||c||c||c||c||c||c||c||c||c||c||c||}
\hline
IC&$\textbf{D}_{\varphi\psi}$&Dim&$\textbf{D}'_{\varphi\psi}$&Dim\\
			\hline
$\mathcal{A}lg_2^1$&
$\left(\begin{array}{cccc}
d_{11}&d_{12}\\
0&d_{11}
\end{array}
\right)$
&
2
&
$\left(\begin{array}{cccc}
0&0\\
0&d'_{22}
\end{array}
\right)$
&
1
\\ \hline 
$\mathcal{A}lg_2^2$&
$\left(\begin{array}{cccc}
d_{11}&d_{12}\\
0&d_{11}
\end{array}
\right)$
&
2
&
$\left(\begin{array}{cccc}
0&0\\
0&d'_{22}
\end{array}
\right)$
&
1
\\ \hline
$\mathcal{A}lg_2^3$&
$\left(\begin{array}{cccc}
d_{11}&d_{12}\\
0&d_{11}
\end{array}
\right)$
&
2
&
$\left(\begin{array}{cccc}
0&0\\
0&d'_{22}
\end{array}
\right)$
&
1
\\ \hline
$\mathcal{A}lg_2^4$&
$\left(\begin{array}{cccc}
d_{11}&d_{12}\\
0&d_{11}
\end{array}
\right)$
&
2
&
$\left(\begin{array}{cccc}
0&0\\
0&d'_{22}
\end{array}
\right)$
&
1
\\ \hline
\end{tabular}
\end{center}
    \item  $3-$dimensional Quasi derivation on BiHom-associative dialgebra:\begin{center}
    \begin{tabular}{||c||c||c||c||c||c||c||c||c||c||c||c||}
\hline
IC&$\textbf{D}_{\varphi\psi}$&Dim&$\textbf{D}'_{\varphi\psi}$&Dim\\
			\hline
$\mathcal{A}lg_3^1$&
$\left(\begin{array}{cccc}
d_{11}&d_{12}&0\\
0&d_{11}&0\\
0&0&d_{33}
\end{array}
\right)$
&
3
&
$\left(\begin{array}{cccc}
0&d'_{12}&0\\
0&0&0\\
0&0&d'_{33}
\end{array}
\right)$
&
2
\\ \hline 
$\mathcal{A}lg_3^2$&
$\left(\begin{array}{cccc}
d_{11}&d_{12}&0\\
0&d_{11}&0\\
0&0&d_{33}
\end{array}
\right)$
&
3
&
$\left(\begin{array}{cccc}
0&d'_{12}&0\\
0&0&0\\
0&0&d'_{33}
\end{array}
\right)$
&
2
\\ \hline
$\mathcal{A}lg_3^3$&
$\left(\begin{array}{cccc}
d_{11}&d_{12}&0\\
0&d_{11}&0\\
0&0&d_{33}
\end{array}
\right)$
&
3
&
$\left(\begin{array}{cccc}
0&d'_{12}&0\\
0&0&0\\
0&0&d'_{33}
\end{array}
\right)$
&
2
\\ \hline
$\mathcal{A}lg_3^4$&
$\left(\begin{array}{cccc}
d_{11}&d_{12}&0\\
0&d_{11}&0\\
0&0&d_{33}
\end{array}
\right)$
&
3
&
$\left(\begin{array}{cccc}
0&d'_{12}&0\\
0&0&0\\
0&0&d'_{33}
\end{array}
\right)$
&
2
\\ \hline
$\mathcal{A}lg_3^5$&
$\left(\begin{array}{cccc}
d_{11}&d_{12}&0\\
0&d_{11}&0\\
0&0&d_{33}
\end{array}
\right)$
&
3
&
$\left(\begin{array}{cccc}
0&d'_{12}&0\\
0&0&0\\
0&0&d'_{33}
\end{array}
\right)$
&
2
\\ \hline
\end{tabular}
\end{center}
\end{itemize}
\subsection{The Classification of Generalized derivation of BiHom-associative dialgebras} In this subsection we define the Generalized-derivation  and provide the classification for $2$ and $3$ dimensional BiHom-associative dialgebra.
\begin{definition}\label{rd3.10}
Let $({A}, \dashv, \vdash, \varphi, \psi)$ be a BiHom-associative dialgebra. A linear map ${D} : {A}\rightarrow {A}$ is called a 
$(\varphi^k, \psi^l)$-derivation of ${A}$, if there exist a derivation map $D:A\to A$ such that the following equations hold
\begin{enumerate}
	\item [$1.$] ${D}\circ\varphi= \varphi\circ{D},\,{D}\circ\psi=\psi\circ{D}$, ${D'}\circ\varphi= \varphi\circ{D'},\,{D'}\circ\psi=\psi\circ{D'}$,\,
	${D''}\circ\varphi= \varphi\circ{D''},\,{D''}\circ\psi=\psi\circ{D''}$
	\item [$2.$] ${D''}(x\dashv y)=\varphi^k\psi^l(x)\dashv{D'}(y)+{D}(x)\dashv\varphi^k\psi^l(y),$
	\item [$3.$] ${D''}(x\vdash y)=\varphi^k\psi^l(x)\vdash {D'}(y)+{D}(x)\vdash\varphi^k\psi^l(y),$
\end{enumerate}
for $x, y\in   {A}.$\end{definition}
We denote by  $GDer_{(\varphi^k, \psi^l)}({A})$, the set of all $(\varphi^k,\psi^l)$-generalized derivations of ${A}$. The space of all generalized derivations of $A$ is given by
$GDer({A}):=\displaystyle\bigoplus_{k\geq 0}\displaystyle\bigoplus_{l\geq 0}GDer_{(\varphi^k, \psi^l)}({A})$.
\begin{itemize}
\item $2-$dimensional Generalized-derivation on BiHom-associative dialgebra:

\begin{center}
    \begin{tabular}{||c||c||c||c||c||c||c||c||c||c||c||c||}
\hline
IC&$\textbf{D}_{\varphi\psi}$&Dim&$\textbf{D}'_{\varphi\psi}$&Dim&$\textbf{D}^{''}_{\varphi\psi}$&Dim\\
			\hline
$\mathcal{A}lg_2^1$&
$\left(\begin{array}{cccc}
d_{11}&d_{12}\\
0&d_{11}
\end{array}
\right)$
&
2
&
$\left(\begin{array}{cccc}
0&d'_{12}\\
0&d'_{22}
\end{array}
\right)$
&
2
&
$\left(\begin{array}{cccc}
d''_{11}&d''_{12}\\
0&d''_{22}
\end{array}
\right)$
&
3
\\ \hline
$\mathcal{A}lg_2^2$&
$\left(\begin{array}{cccc}
d_{11}&d_{12}\\
0&d_{11}
\end{array}
\right)$
&
2
&
$\left(\begin{array}{cccc}
0&d'_{12}\\
0&d'_{22}
\end{array}
\right)$
&
2
&
$\left(\begin{array}{cccc}
d''_{11}&d''_{12}\\
0&d''_{22}
\end{array}
\right)$
&
3
\\ \hline
$\mathcal{A}lg_2^3$&
$\left(\begin{array}{cccc}
d_{11}&d_{12}\\
0&d_{11}
\end{array}
\right)$
&
2
&
$\left(\begin{array}{cccc}
0&d'_{12}\\
0&d'_{22}
\end{array}
\right)$
&
2
&
$\left(\begin{array}{cccc}
d''_{11}&d''_{12}\\
0&d''_{22}
\end{array}
\right)$
&
3
\\ \hline
$\mathcal{A}lg_2^4$&
$\left(\begin{array}{cccc}
d_{11}&d_{12}\\
0&d_{11}
\end{array}
\right)$
&
2
&
$\left(\begin{array}{cccc}
0&d'_{12}\\
0&d'_{22}
\end{array}
\right)$
&
2
&
$\left(\begin{array}{cccc}
d''_{11}&d''_{12}\\
0&d''_{22}
\end{array}
\right)$
&
3
\\ \hline
\end{tabular}

\end{center}
\item $3-$dimensional Generalized-derivation on BiHom-associative dialgebra:
\begin{center}
    \begin{tabular}{||c||c||c||c||c||c||c||c||c||c||c||c||}
\hline
IC&$\textbf{D}_{\varphi\psi}$&Dim&$\textbf{D}'_{\varphi\psi}$&Dim&$\textbf{D}^{''}_{\varphi\psi}$&Dim\\
			\hline
$\mathcal{A}lg_3^1$&
$\left(\begin{array}{cccc}
d_{11}&d_{12}&0\\
0&d_{11}&0\\
0&0&d_{33}
\end{array}
\right)$
&
3
&
$\left(\begin{array}{cccc}
0&d'_{12}&d'_{13}\\
0&d'_{22}&d'_{23}\\
0&d'_{32}&d'_{33}
\end{array}
\right)$
&
6
&
$\left(\begin{array}{cccc}
d''_{11}&d''_{12}&0\\
0&d''_{22}&0\\
0&0&d''_{33}
\end{array}
\right)$
&
4
\\ \hline 
$\mathcal{A}lg_3^2$&
$\left(\begin{array}{cccc}
d_{11}&d_{12}&0\\
0&d_{11}&0\\
0&0&d_{33}
\end{array}
\right)$
&
3
&
$\left(\begin{array}{cccc}
0&d'_{12}&d'_{13}\\
0&d'_{22}&d'_{23}\\
0&d'_{32}&d'_{33}
\end{array}
\right)$
&
6
&
$\left(\begin{array}{cccc}
d''_{11}&d''_{12}&0\\
0&d''_{22}&0\\
0&0&d''_{33}
\end{array}
\right)$
&
4
\\ \hline 
$\mathcal{A}lg_3^3$&
$\left(\begin{array}{cccc}
d_{11}&d_{12}&0\\
0&d_{11}&0\\
0&0&d_{33}
\end{array}
\right)$
&
3
&
$\left(\begin{array}{cccc}
0&d'_{12}&d'_{13}\\
0&d'_{22}&d'_{23}\\
0&d'_{32}&d'_{33}
\end{array}
\right)$
&
6
&
$\left(\begin{array}{cccc}
d''_{11}&d''_{12}&0\\
0&d''_{22}&0\\
0&0&d''_{33}
\end{array}
\right)$
&
4
\\ \hline  
$\mathcal{A}lg_3^4$&
$\left(\begin{array}{cccc}
d_{11}&d_{12}&0\\
0&d_{11}&0\\
0&0&d_{33}
\end{array}
\right)$
&
3
&
$\left(\begin{array}{cccc}
0&d'_{12}&d'_{13}\\
0&d'_{22}&d'_{23}\\
0&d'_{32}&d'_{33}
\end{array}
\right)$
&
6
&
$\left(\begin{array}{cccc}
d''_{11}&d''_{12}&0\\
0&d''_{22}&0\\
0&0&d''_{33}
\end{array}
\right)$
&
4
\\ \hline 
$\mathcal{A}lg_3^5$&
$\left(\begin{array}{cccc}
d_{11}&d_{12}&0\\
0&d_{11}&0\\
0&0&d_{33}
\end{array}
\right)$
&
3
&
$\left(\begin{array}{cccc}
0&d'_{12}&d'_{13}\\
0&d'_{22}&d'_{23}\\
0&d'_{32}&d'_{33}
\end{array}
\right)$
&
6
&
$\left(\begin{array}{cccc}
d''_{11}&d''_{12}&0\\
0&d''_{22}&0\\
0&0&d''_{33}
\end{array}
\right)$
&
4
\\ \hline 
\end{tabular}\\
\end{center}
\end{itemize}
\section{Cohomology of BiHom-associative dialgebra} In this section, we study the Cohomology of BiHom-associative dialgebra. To accomplish this task, we first recall the cohomology of BiHom-associative algebra.
\subsection{Cohomology of BiHom-associative dialgebras.}\,
\begin{definition}
Let $(A, \mu, \varphi, \psi)$ be a BiHom-associative algebra. For each $n\geq 1$,
we define a vector space $\mathbb{C}^n_{Hoch}(A,A)$ consisting of all multilinear maps  
$f : A^{\otimes n}\longrightarrow A$ satisfying  
\begin{equation}\label{com1}
\varphi\circ f=f\circ\varphi^{\otimes n}\quad\text{and}\quad \psi\circ f=f\circ\psi^{\otimes n}.
\end{equation}
Define a coboundary map $\delta_{Hoch}(A,A): \mathbb{C}^{n}_{Hoch}\longrightarrow \mathbb{C}^{n+1}_{Hoch}$, such that for any $f\in \mathbb{C}^{n}_{Hoch}$ and $a_1, a_2,\cdots, a_{n+1}\in A,$ we have
\begin{equation*}
\begin{array}{ll}
(\delta_{Hoch}f)(a_1,\cdots,a_{n+1})
=&\mu(\varphi^{n-1}(a_1),f(a_2,\cdots,a_{n+1}))\\
&+\sum^n_{i=1}(-1)^if(\varphi(a_1),\cdots,\varphi(a_{i-1}),\mu(a_i,a_{i+1}),\psi(a_{i+2}),\cdots,\psi(a_{n+1}))\\
&+(-1)^{n+1}\mu(f(a_1,\cdots,a_n),\psi^{n-1}(a_{n+1})),
\end{array}
\end{equation*}
\end{definition}
These groups  correspond in Deformation theory to the space of all obstructions to extend a deformation of order $n$ to a deformation of order $n+1$.
A $3$-coboundary operator of a BiHom-associative algebra $A$ is thus given by a  map 
\begin{align*}
    \delta^3_{Hoch} : \mathbb{C}^3(A,A)&\rightarrow \mathbb{C}^4(A, A), \\ \Psi&\mapsto\delta^3_{Hoch}\Psi 
\end{align*}
defined by 
$$\begin{array}{ll} 
\delta^3_{Hoch}\Psi(x,y,z,w)
=&\mu(\varphi^2(x),\psi(y,z,w))-\Psi(\mu(x,y),\psi(z),\psi(w))+\psi(\varphi(x),\mu(y,z), \psi(w))\\
&-\Psi(\varphi(x),\varphi(y),\mu(z,w))
+\mu(\Psi(x,y,z)),\psi^2(w)).
\end{array}$$ 

\begin{proposition}
$\delta^3_{Hoch}(\delta^2_{Hoch})=0.$
\end{proposition}
\begin{proof}
Recall that
$$ \delta^2_{Hoch}f(x,y,z)
=\mu(\varphi(x),f(y,z))-f(\mu(x,y),\psi(z))
+f(\varphi(x),\mu(y,z))-\mu(f(x,y),\psi(z)).$$    Subsequently, the existence of 
\begin{center}
   $ \begin{array}{ll} 
&\delta^3_{Hoch}(\delta^2_{Hoch}f)\psi(x,y,z,w)=\\
=&\mu(\varphi^2(x),\delta^2_{Hoch}f(y,z,w))-\delta^2_{Hoch}f(\mu(x,y),\psi(z),\psi(w))+\delta^2_{Hoch}f(\varphi(x),\mu(y,z),\psi(w))\\
&-\delta^2_{Hoch}f(\varphi(x),\varphi(y),\mu(z,w))+\mu(\delta^2_{Hom}f(x,y,z)),\psi^2(w))\\
=&\mu(\varphi^2(x),\mu(\varphi(y),f(z,w)))-\mu(\varphi^2(x),f(\mu(y,z),\psi(w)))+\mu(\varphi^2(x),f(\varphi(y),\mu(z,w)))\\
&-\mu(\varphi^2(x),\mu(f(y,z),\psi(w)))+\mu(\varphi\circ\mu(x, y),f(\psi(z),\psi(w)))-f(\mu(\mu(x, y), \psi(z),\psi^2(w))\\
&+f(\varphi\circ\mu(x, y),\mu(\psi(z), \psi(w)))-\mu(f(\mu(y,z),\psi(z)),\psi^2(w))+\mu(\varphi^2(x), f(\mu(y,z),\psi(w)))\\
&-f(\mu(\varphi(x),\mu(y,z)),\psi^2(w))+f(\varphi^2(x),\mu(\mu(y,z),\psi(w)))-\mu(f(\varphi(x),\mu(y,z)),\psi^2(w)))\\
&+\mu(\varphi^2(x),f(\varphi(y),\mu(z,w)))-f(\mu(\varphi(x),\varphi(y)),\psi\circ\mu(z,w))+f(\varphi^2(x),\mu(\varphi(y),\mu(z,w)))\\
&-\mu(f(\varphi(x),\varphi(y)),\psi\circ\mu(z,w))+\mu(\mu(\varphi(x),f(y,z)),\psi^2(w))-\mu(f(\mu(x,y),\psi(z)),\psi^2(w))\\
&+\mu(f(\varphi(x),\mu(y,z)),\psi^2(w))-\mu(\mu(f(x,y),\psi(z)),\psi^2(w))\\
=&0,
\end{array}$
\end{center} can be established due to the commutativity of $\varphi$, $\psi$, and $f$, as well as the BiHom-associativity of the multiplication operator $\mu$.
\end{proof}
\begin{example}We consider the $3$-dimensional BiHom-associative algebra  with a basis $\left\{e_1, e_2, e_3\right\}$, defined by\begin{center}
     $e_1\cdot e_2=e_1,$~ $ e_1\cdot e_2=e_2,$~ $ e_2\cdot e_1=e_2,$\\~ $ e_2\cdot e_2=e_2,$~ $e_3\cdot e_2=e_3,$ ~$e_3\cdot e_3=e_3,$\\ $\varphi(e_2)=e_2,$ $\psi(e_1)=e_1,$~ $\psi(e_2)=e_1-e_2.$
\end{center}
\begin{enumerate}
    \item \noindent We have the following  $2$-cocycles : \begin{eqnarray*}
    f(e_1, e_3)= a_{13}^3e_3,\quad f(e_2, e_3)= a_{23}^3e_3,\quad f(e_3, e_3)=a_{33}^3e_3.
    \end{eqnarray*}
\item \noindent We have the following $3$-cocycles:
\begin{eqnarray*}
    \begin{aligned}
\psi(e_1,e_1,e_1)=&a_{111}^3e_3,\quad
\psi(e_1,e_2,e_3)=&a_{123}^3e_3,\\
\psi(e_1,e_3,e_1)=&a_{131}^3e_3,\quad
\psi(e_1,e_3,e_2)=&a_{132}^3e_3,\\ 
\psi(e_1,e_3,e_3)=&a_{133}^3e_3,\quad
\psi(e_2,e_1,e_3)=&a_{213}^3e_3,\\
\psi(e_2,e_3,e_3)=&a_{233}^3e_3,\quad
\psi(e_3,e_1,e_3)=&a_{313}^3e_3,\\
\psi(e_3,e_2,e_3)=&a_{323}^3e_3,\quad
\psi(e_3,e_3,e_1)=&a_{331}^3e_3,\\
\psi(e_3,e_3,e_1)=&a_{332}^3e_3,\quad
\psi(e_3,e_3,e_1)=&a_{332}^3e_3.
    \end{aligned}
\end{eqnarray*}
\end{enumerate}
\end{example}
\par Now we  develop the cochain complex and cohomology of BiHom-associative dialgebras by using techniques of  \cite{AF, AD, AD1}.

Let $A=(A, \dashv,\vdash , \varphi, \psi)$ be a BiHom-associative dialgebra. We define the cohomology of $A$  by using planar binary trees, thus our cohomology will generalize the cohomologies of associative dialgebras \cite{AF}, Hom-associative dialgebras \cite{AD} and BiHom-associative dialgebras \cite{AD1}.

Let $Y_n$ be the set of all planar binary trees with $n$-vertices, $(n+1)$ leaves, and one root in such a way that each vertex is trivalent.  Take $Y_0$ to be the singleton set consisting of only one root. Thus, for small values of $n$, we have 
   \[
   Y_0 = \left\lbrace \tone \, \right\rbrace,~
   Y_1 = \left\lbrace \ttwo \right\rbrace,~
   Y_2 = \left\lbrace \tthreeone,\, \tthreetwo \right\rbrace,~
   Y_3 = \left\lbrace \tfourone,\, \tfourtwo,\, \tfourthree,\, \tfourfour,\, \tfourfive \right\rbrace.
   \]

For any three $y\in Y_n$, we label the $(n+1)$ leaves of $y$ by $\left\{0, 1,2,\cdots,n\right\}$ from left to right. (There is a grafting operation on the space of all planar binary trees). More, precisely, for  any $y_1\in Y_n$ and $y_2\in Y_n$, the grafting of $y_1$ and $y_2$ is the tree $y_1\vee y_2\in Y_{n+m+1}$ obtained by joining the roots of $y_1$ and $y_2$ and creating a root from that vertex.  For any $0\leq i\leq n+1$, there is a map $\bullet_i : Y_{n+1}\longrightarrow \left\{\dashv, \vdash\right\}$ define by
\begin{center}
$\bullet_0(y)=\bullet_0^y=
\left\{\begin{array}{c}
\dashv \, \text{if}\,  y\,  \text {has the shape}\,  |\vee y_1\,  \text{for some}\, y_1\in Y_n,\\
\vdash\, \text{otherwise}, 
\end{array}\right.
$ \\     
$\bullet_i(y)=\bullet_i^y=
\left\{\begin{array}{c}
\dashv \, \text{if}\,  \text {the\, i-th\, leaf\, of}\, y\,  \text{is\,  oriented\, like}\, '\backslash', \\
\vdash\,  \text{if}\,  \text {the\, i-th\,  leaf\,  of}\, y\,  \text{is\,  oriented\,  like}\, '/',
\end{array}\right.
$  \\ 
$\bullet_{n+1}(y)=\bullet_{n+1}^y=
\left\{\begin{array}{c}
\dashv \, \text{if}\,  y\,  \text {has the shape}\,  y_1\vee |\,  \text{for\,  some} ~y_1\in Y_n\\
\vdash\, \text{otherwise}, 
\end{array}\right.
$  
\end{center} 

Finally, for any $n\geq 1$, there are maps $d_i : Y_n \longrightarrow  Y_{n-1}, (0\leq i\leq n)$ $y\longmapsto d_i(y)$, where $d_i(y)$ is obtained from $y$ by removing the $i-th$ leaf of $y.$ This map is also named a face map. The face maps satisfy the relation $d_id_j = d_{j-1}d_i$, for all $i<j$. Now we are in a position to define the cohomology of a BiHom-associative dialgebra $A= (A, \dashv, \vdash, \varphi, \psi)$.

\begin{definition}
 For each $n\geq 0$, we define an abelian group 
 $ CY_{\varphi, \psi}^n(A, A):=Hom(\mathbb{K}\left[Y_n\right]\otimes A^{\otimes n}$, A) and define a map 
$\delta_{BiH} : CY_{\varphi, \psi}^n(A, A)\longrightarrow  CY_{\varphi, \psi}^{n+1}(A, A)$ by 

\begin{equation*}
\begin{aligned}
    (\varphi\circ f)(y ; a_1,\cdots,a_{n})=&f(y; \varphi(a_1),\cdots,\varphi(a_n))\\\text{and}\\(\psi\circ f)(y ; a_1,\cdots,a_{n})=&f(y; \psi(a_1),\cdots,\psi(a_n))
\end{aligned}
\end{equation*}

\begin{equation}
\begin{array}{ll}\label{com3}
(\delta_{BiH}^nf)(y ; a_1,\cdots,a_{n+1})
&=\varphi^{n-1}(a_1)\bullet_0^yf(d_o(y) ; a_2\cdots,a_{n+1})\\
&+\sum^n_{i=1}(-1)^if(d_i(y);\alpha(a_1),\cdots, a_i\bullet_i^ya_{i+1}, \psi(a_{i+2}), \cdots, \psi(a_{n+1}))\\
&+(-1)^{n+1}f(d_{n+1}(y) ; a_1,\cdots,a_n)\bullet_{n+1}^y\psi^{n-1}(a_{n+1}),
\end{array}\end{equation}
for $y\in Y_{n+1}$ and $a_1, \cdots, a_{n+1}\in A.$ 
\end{definition}Thus an abelian group $ CY_{\varphi, \psi}^n(A, A)$ along with the coboundary operator $\delta_{BiH}$ is called a cochain complex $\left\{CY^n(A, A), \delta_{BiH}\right\}$. The cohomology corresponding to this cochain complex is called the Hochschild cohomology of the BiHom-associative dialgebra $A.$ We denote it by $HY^n(A, A)$.

Let $A$ be a vector space and $\varphi,\psi: A\rightarrow A$ be linear  maps. For each $k\geq 1$ define $CY_{\varphi\psi}^k(A,A)$ be the space of all multilinear maps $f : K[Y_k]\otimes A^{\otimes k}\longrightarrow A$ satisfying 
\begin{center}
    $(\varphi\circ f)(y ; a_1,\cdots,a_{n})=f(y; \varphi(a_1),\cdots,\varphi(a_n)),$\\$(\psi\circ f)(y ; a_1,\cdots,a_{n})=f(y;\psi(a_1),\cdots,\psi(a_n))$
\end{center}
for all $y\in Y_k$ and $a_i\in A$. 
 
Our aim is to define on operad structure on $\Theta=\left\{\Theta(k)|k\geq 1\right\}$ where $\Theta(k)=CY_{\varphi\psi}^k(A,A)$, for $k\geq 1.$
For this, we closely follow \cite{AG}. 
For any $k, n_1,\cdots, n_k\geq 1$, we define maps
$$R_0(k, n_1,\cdots, n_k)=d_1,\cdots, d_{n_1+\cdots+n_{k-1}-1}d_{n_1+\cdots+n_{k-1}+1}\cdots d_{n_1 +1}\cdots d_{n_1+\cdots n_{k} -1}.$$
Moreover, for any $1\leq i\leq k$, there are maps $R_i(k; n_1,\cdots, n_k) : Y_{n_1+\cdots+n_k}\rightarrow Y_{n_i}$ define by 
$$R_i(k; n_1,\cdots, n_k)=d_0d_1\cdots d_{n_1+n_{i-1}-1}d_{n_1+\cdots+n_{i+1}}\cdots d_{n_1+\cdots n_k}.$$
In the other words, the function $R_0(k; n_1,\cdots, n_k)$ misses $d_0, d_{n_1},d_{n_1+n_2},\cdots,d_{n_1+\cdots +n_k}$ and the function 
$R_i(k; n_1,\cdots, n_k)$ misses $d_{n_1},d_{n_{i-1}},d_{n_1},d_{n_{i-1}+1},\cdots,d_{n_1+\cdots +n_i}.$
$$R=\left\{R_0(k; n_1,\cdots, n_k),R_i(k; n_1,\cdots, n_k)|k,n_1,\cdots, n_k\geq1\, \text{and}\,1\leq i\leq k\right\}$$ satisfies the following relations of a pre -operadic 
system \cite{Y1}.
\begin{enumerate}
	\item[$\bullet$] $R_0(k, \underbrace{1,\cdots,1}_{k\, times})=id_{Y_k}$ for each $k\geq 1$
	\item[$\bullet$] $R_0(k, n_1,\cdots,n_k)R_0(n_{1}+\cdots+n_k; m_1,\cdots, m_{n_1+\cdots+n_k})$
	\item[$=$]$R_0(k, m_1,\cdots+m_{n_1}, m_{n_1+1}+\cdots+m_{n_1+n_2},\cdots, m_{n_1+\cdots+n_{k-1}+1}+\cdots+m_{{n_1}+\cdots+n_k})$,
	\item[$\bullet$]$R_i(k, n_1,\cdots,n_k)R_0(n_{1}+\cdots+n_k; m_1,\cdots, m_{n_1+\cdots+n_k})=R_0(n_i ; m_{n_1+\cdots+n_{i-1}+1},\cdots, m_{n_1+\cdots+n_i}),$\\
$R_i(k; m_1,\cdots, m_{n_1},m_{n_1+1}+\cdots+m_{n_1+n_2},\cdots,m_{n_1+\cdots+n_{k-1}+1}+\cdots+ m_{n_1+\cdots+n_k}),$
\item[$\bullet$]$R_{n_1,\cdots,n_{i-1}+j})(n_{1}+\cdots+n_k; m_1,\cdots, m_{n_1+\cdots+n_k})=R_j(n_i ; m_{n_1+\cdots+n_{i-1}+1},\cdots, m_{n_1+\cdots+n_i})$\\
$R_i(k; m_1,\cdots, m_{n_1},m_{n_1+1}+\cdots+m_{n_1+n_2},\cdots,m_{n_1+\cdots+n_{k-1}+1}+\cdots+ m_{n_1+\cdots+n_k}),$
\end{enumerate}
for any $ m_1,\cdots, m_{n_1+\cdots+n_k}\geq 1.$
\par At this point, we have the capability to establish an operad structure on $\Theta.$ Let's proceed with defining the following partial composition
$$\circ_i : \Theta(m)\otimes\Theta(n)\longrightarrow \Theta(m+n-1)$$ by 

$$\begin{array}{ll} 
(f\circ_i g)(y ; a_1\cdots,a_{m+n-1})
&=f\Big(R_0(m; \overbrace{1,\cdots,1,\underbrace{n}_{i-\text{th place}}1\cdots,1}^{m-\text{tuple}})y; \varphi^{n-1}a_1,\cdots,\varphi^{n-1}a_{i-1},\\
&\quad g(R_i(m; \overbrace{1,\cdots,1,\underbrace{n}_{i-\text{th place}}1\cdots,1}^{m-\text{tuple}})y;a_i,\cdots,a_{i+n-1}),\psi^{n-1}a_{i+n},\cdots,\psi^{n-1}a_{m+n+1})\Big).
\end{array}$$

for $f\in CY^{m}_{\varphi\psi}(A,A), g\in  CY^{n}_{\varphi\psi}(A,A),y\in Y_{m+n-1}$ and $a_1,\cdots,a_{m+n-1}\in A.$ The two definitions of non-$\sum$
operad are related by 

\begin{equation}
\begin{array}{ll}
(f\circ_i g)=y(f; \overbrace{id,\cdots,id,\underbrace{g}_{i-\text{th place}}id\cdots,id}^{m-\text{tuple}}),\quad {for }\quad f\in \Theta(m)
\end{array}
\end{equation}

\begin{equation}
\begin{array}{ll}
\gamma(f ; g_1,\cdots, g_k)(\cdots((f\circ_kg_k)\circ_{k-1\circ}g_{k-1},\cdots)\circ_1g_1,\quad {for }\quad f\in \Theta(k).
\end{array}
\end{equation}

The pre-operaic identities, it follows that the compositions 

$$\begin{array}{ll}
&\gamma_{\varphi\psi}(f;g_1,\cdots,g_k)(y; a_1,\cdots,a_{n_1+\cdots+n_k})\\
&=f\Bigl(R_0(k,n_1,\cdots,n_k)y;\varphi^{\sum^{k}_{i=2}\left|g_i\right|}\cdot g_1\Bigl(R_1(k;n_1,\cdots,n_k)y;a_1,\cdots,a_{n_1}\Bigr),\cdots, \\
&\qquad \varphi^{\sum^{k}_{i=2}\left|g_i\right|}\cdot\psi^{\sum_{i=2, l\neq i}^k \left|g_l\right|} g_i\Bigl(R_i(k;n_1,\cdots,n_k)y;\bigl(a_{n_1+\cdots+n_{i-1}+1}, \cdots , a_{n_1+\cdots+n_i}\bigr)\Bigr),\\
&\qquad \psi^{\sum_{i=1}^k \left|g_i\right|} g_k\Bigl(R_k(k;n_1,\cdots,n_k)y;a_{n_1+\cdots+n_{k-1}+1},\cdots, a_{n_1+\cdots+n_k}\Bigr)\Bigr),
\end{array}$$
for all $y\in Y_{n_1+\cdots+n_k}$ and $a_1,a_2,\cdots,a_{n_1+\cdots+n_k}\in A.$

We also consider the identity map $id\in CY_{\varphi\psi}^1(A,A)$ defined by $id([1] ; a)=a$, for all $a\in A.$

Note that, the corresponding braces are given by 
$$
\left\{f\right\}\left\{g_1,\cdots,g_n\right\}=\sum(-1)^\varepsilon\gamma_{\varphi\psi}(f; id,\cdots,id,g_1,id,\cdots,id,g_n,id,\cdots,id).
$$

the degre $-1$ graded Lie bracket on $CY_{\varphi\psi}^\bullet(A,A)$ is given by $$\left[f,g\right]=f\circ g-(-1)^{(m-1)(n-1)}g\circ f,$$
where 
$$\begin{array}{ll}
&(f\circ g)(y; a_1,\cdots,a_{m+n-1})\\
&=\sum_{i=1}^m(-1)^{(i-1)(n-1)}f\Big(R_0(m; \overbrace{1,\cdots,1,\underbrace{n}_{i-\text{th place}}1\cdots,1}^{m-\text{tuple}})y; \varphi^{n-1}a_1,\cdots,\varphi^{n-1}a_{i-1},\\
&\hspace{3cm}g\Big(R_i(m; \overbrace{1,\cdots,1,\underbrace{n}_{i-\text{th place}}1\cdots,1}^{m-\text{tuple}})y; a_i,\cdots,a_{i+n-1}),\psi^{n-1}a_{i+n},\cdots,\psi^{n-1}a_{m+n-1}\Big),
\end{array}$$
for $f\in CY_{\varphi\psi}^m(A,A), g\in CY_{\varphi\psi}^m(A,A)$ and $a_1,\cdots,a_{m+n-1}\in A.$

Next, let $(A, \dashv, \vdash, \varphi, \psi)$ be a BiHom-dialgebra. Consider the operad structure on $CY^\bullet_{\varphi\psi}(A,A)$ as defined above. Define an element 
$\pi\in CY_{\varphi\psi}^2(A,A)$ by the following 

$$
\pi(y;a,b)=
\left\{\begin{array}{c}a\vdash b\quad \text{if}\quad y=\left[12\right],\\  
a\dashv b\quad \text{if}\quad y=\left[21\right],
\end{array}\right.
$$
for all $a,b\in A.$ An easy calculation shows that 

$$
\left\{\pi\right\}\left\{\pi\right\}(y;a,b,c)=
\left\{\begin{array}{c}  
(a\vdash b)\vdash\psi(c)-\varphi(a)\vdash(b\vdash c)\quad \text{if}\, y=\left[123\right],\\
(a\dashv b)\vdash\psi(c)-\varphi(a)\vdash(b\vdash c)\quad \text{if}\, y=\left[213\right],\\
(a\vdash b)\dashv\psi(c)-\varphi(a)\vdash(b\dashv c)\quad \text{if}\, y=\left[131\right],\\
(a\dashv b)\dashv\psi(c)-\varphi(a)\dashv(b\vdash c)\quad \text{if}\, y=\left[312\right],\\
(a\dashv b)\dashv\psi(c)-\varphi(a)\dashv(b\dashv c)\quad \text{if}\, y=\left[321\right].
\end{array}\right.
$$
Hence, it follows from the BiHom-dialgebra condition that $\left\{\pi\right\}\left\{\pi\right\}(y; a, b,c)=0,$ for all $y\in Y_3$ and 
$a, b, c\in A.$  Therefore, $\pi$ define a multiplication on the operad $CY_{\varphi\psi}^\bullet(A,A).$ The corresponding dot product on $CY_{\varphi\psi}^\bullet(A,A)$ is given by   
    $$\begin{array}{ll}
&(f\cdot g)(y; a_1,\cdots,a_{m+n})\\&=(-1)^m\left\{\pi\right\}\left\{f,g\right\}(y;a_1,\cdots,a_{m+n})\\
&=(-1)^{mn}\pi\left(R_0(2;m,n)y; \varphi^{n-1}f\left(R_1(2;m,n)y; a_1,\cdots,a_m\right),\psi^{m-1}g\left(R_2(2;m,n)y;a_{m+1},\cdots,a_{m+n}\right)\right),
\end{array}$$
for $f\in CY^m_{\varphi\psi}(A,A), g\in CY^m_{\varphi\psi}(A,A), y\in Y_{m+n}$ and $a_1,\cdots,a_{m+n}\in A.$

 \section{Formal deformations of BiHom-associative dialgebras}
In this section, we aim to discuss the geometric classification  of BiHom-associative dialgebras. To accomplish our objective, we use one parameter formal deformation theory that was initially proposed by Gerstenhaber for associative algebras and later extended to BiHom-associative algebras.
\begin{definition}
Let $(\Theta, \eta, id)$ be an operad. Consider the space $\Theta(n)[[t]]$ of formal power series in a variable t with coefficients in $\Theta(n)$. One can linearly extend the circle products to
$\displaystyle\bigoplus_{n\geq 0}\Theta(n)[[t]].$ 
\end{definition}Let $\pi$ be a fixed multiplication on $\Theta.$ 
\begin{definition}
A formal 1-parameter deformation of $\pi$ is given by a formal sum $\pi_t=\pi_0+\pi_1t+\pi_2t^2+\cdots\in \Theta(2)[[t]]$ with $\pi_0=\pi$, satisfying $\pi_t\circ\pi_t=0.$ This is equivalent to a system
of equations  :
\begin{equation}\label{eqd}
\sum_{i+j=n}\pi_i\circ\pi_j=0,\quad for\quad  n\geq 0.
\end{equation}
\end{definition}
Moreover, if $(A, \dashv, \vdash, \varphi, \psi)$ is a BiHom-associative dialgebra, then the element $\pi\in \Theta(2)$ given by $\pi( \tthreeone ; a, b)=a\dashv b$ and $\pi(\tthreetwo ; a, b)=a\vdash b$
defines a multiplication on $\Theta.$ A deformation of BiHom-associative dialgebra $(A, \dashv, \vdash, \varphi, \psi)$ in the sense of \cite{AG} is given by two formal power series 
$\dashv_t=\dashv_0+\dashv_1t+\dashv_2t^2+\cdots$  and $\dashv_t=\vdash_0+\vdash_1t+\vdash_2t^2+\cdots$  $(\dashv_0=\dashv$ and $\vdash_0=\vdash)$ of binary operations on A such that 
$(A[[t]], \dashv_t, \vdash_t, \varphi, \psi)$ forms a BiHom-associative dialgebra over $\mathbb{K}[[t]].$ Then it follows that the the formal sum $\pi_t=\pi_0+\pi_1t+\pi_2t^2+\cdots\in \Theta(2)[[t]]$
defines a deformation of the multiplication $\pi$, where $\pi_i( \tthreeone ; a, b)=a\dashv_i b$ and $\pi_i(\tthreetwo ; a, b)=a\vdash_i b$ for all $i\geq0.$This follows that the deformation equations of \cite{AG}
are equivalent to the deformation equation (\ref{eqd}).
\begin{definition}
Let $(A,\dashv, \vdash, \varphi, \psi)$ be a  BiHom-associative dialgebra. A $1$-parameter formal  deformation of the  BiHom-associative dialgebra $A$ is given by 
 a $\K[[t]]$-bilinear maps  $\dashv_t, \vdash_t : A[[t]]\times A[[t]]\longrightarrow A[[t]]$ of the form
$\dashv_t=\sum_{i\geq0}t^i \gamma_i$, $\vdash_t=\sum_{i\geq0}t^i\delta_i$ where each $\dashv_i, \vdash_i\in C^2_{Bd}$ with $\dashv_0=\dashv$ and $\vdash_0=\vdash$ such that the following holds 
\begin{eqnarray}
(x\dashv_t y)\dashv_t\psi(z)&=&\varphi(x)\dashv_t(y\dashv_t z),\label{eq14}\\
(x\dashv_t y)\dashv_t\psi(z)&=&\varphi(x)\dashv_t(y\vdash_t z),\label{eq15}
\end{eqnarray}
 for all $x, y, z\in A$.\\
These are equivalent to systems equations: for $n\geq 0,$
\begin{eqnarray*}
\sum_{i+j=n}(x\dashv_jy)\dashv_i\psi(z)&=&\sum_{i+j=n}\varphi(x)\dashv_i(y\dashv_jz),\label{eq19}\\
\sum_{i+j=n}(x\dashv_j y)\dashv_i\psi(z)&=&\sum_{i+j=n}\varphi(x)\dashv_i(y\vdash_j z),\label{eq20}
\end{eqnarray*}
 for all $x, y, z\in A$. 	
\end{definition}
 \begin{proposition}
Let $\dashv_{t}=\psi^{-1}\circ \dashv_t'\circ (\psi\otimes \psi)$, $\vdash_{t}=\psi^{-1}\circ \vdash_t'\circ (\psi\otimes \psi)$, $\varphi_{t}=\psi^{-1}\circ\varphi' \circ\psi$ and 
$\psi_{t}=\psi^{-1}\circ\psi' \circ\psi$. 
If $(A, \dashv_t',\vdash_t', \varphi',\psi')$ is  BiHom-associative dialgebra then $(A\left[\left[t\right]\right], \dashv_{t}, \vdash_{t} , \varphi_{t}, \psi_{t})$ 
is BiHom-associative dialgebra.
\end{proposition}
\begin{proof}
By straightforward computation, we have
\begin{align*}
\varphi_{t}(x)\dashv_{t}(y\dashv_{t} z)
&=\psi^{-1}\circ\varphi'\circ\psi(x)\dashv_{t}\psi^{-1}(\psi(y)\dashv_t'\psi(z))\\
&=\psi^{-1}(\psi^{-1}\circ\psi\circ\varphi'\circ\psi(x)\dashv_t'(\psi^{-1}\circ\psi(\psi(y)\dashv_t'\psi(z))\\
&=\psi^{-1}(\varphi'\circ\psi(x)\dashv_t'(\psi(y)\dashv_t'\psi(z)))\\
&=\psi^{-1}((\psi(x)\dashv_t'\psi(y))\dashv_t'\psi'\circ\psi(z))\\
&=\psi^{-1}(\psi\circ\psi^{-1}(\psi(x)\dashv_t'\psi(y))\dashv_t'\psi\circ\psi^{-1}\psi'\circ\psi(z)))\\
&=(x\dashv_{t}y)\dashv_t\psi_t(z).
\end{align*}
This proves our proposition.
\end{proof}
\begin{definition}
Two deformations $\dashv_{t}=\sum_{i\geq0}\gamma_it^i, \dashv'_{t}=\sum_{i\geq0}\tilde{\gamma}_it^i,$ and $\vdash_{t}=\sum_{i\geq0}\delta_it^i, \vdash'_{t}=\sum_{i\geq0}\tilde{\delta}_it^i$
of the  BiHom-associative dialgebra A said to be equivalent if there exists a formal automorphism\\  $(\psi_t)_{t\geq 0}:  A\left[\left[t\right]\right]\longrightarrow A\left[\left[t\right]\right]$ of the
form $\psi_{t}=\sum_{i\geq0}\psi_it^i$ (where $\psi_i\in C^i_{Bd}(A, A)$ with $\psi_0=id$), 
\begin{eqnarray}
\psi_t(x\dashv_t y)&=&\psi_t(x)\dashv'_t\psi_t(y)\label{eq24}\\
\psi_t(x\vdash_t y)&=&\psi_t(x)\vdash'_t\psi_t(y)\label{eq25}\\
\psi_t(\psi_t(x))=\psi'_t(\psi_t(x))&, &\psi_t(\psi_t(x))=\psi'_t(\psi_t(x))
\end{eqnarray}
for all $x, y\in A$. \\
This deformation $A_t$ of $A$ is said to be trivial if and only if $A_t$ is equivalent to $A$. The identities (\ref{eq24}) and (\ref{eq25}) may be written for all $x, y\in A, $
\begin{eqnarray*}
\sum_{i\geq0,j\geq 0}(\psi_i(x\dashv_jy))t^{i+j}&-&\sum_{i\geq0,j\geq 0,k\geq0}(\psi_i(x)\dashv'_j\psi_k(y))t^{i+j+k}=0,\\
\sum_{i\geq0,j\geq 0}(\psi_i(x\vdash_jy))t^{i+j}&-&\sum_{i\geq0,j\geq 0,k\geq0}(\psi_i(x)\vdash'_j\psi_k(y)t^{i+j+k}=0.
\end{eqnarray*}
In particular,
\begin{eqnarray*}
\psi_0(x\dashv_1 y)+\psi_1(x\dashv_0 y)=\psi_0(x)\dashv'_0\psi_1(y)+\psi_1(x)\dashv'_0\psi_0(y)+\psi_0(x)\dashv'_1\psi_0(y)\\
\psi_0(x\vdash_1 y)+\psi_1(x\vdash_0 y)=\psi_0(x)\vdash'_0\psi_1(y)+\psi_1(x)\vdash'_0\psi_0(y)+\psi_0(x)\vdash'_1\psi_0(y)
\end{eqnarray*}
Since $\psi_0=id$ then 
\begin{eqnarray*}
x\dashv'_1 y=x\dashv_1 y-\psi_1(x\dashv_0y)-x\dashv'_0\psi_1(y)-\psi_1(x)\dashv'_0y\\
x\vdash'_1 y=x\vdash_1 y-\psi_1(x\vdash_0y)-x\vdash'_0\psi_1(y)-\psi_1(x)\vdash'_0y.
\end{eqnarray*}
\end{definition}
\begin{proposition}
Let $\dashv'=\psi_t\circ\dashv(\psi^{-1}_t\otimes\psi_t^{-1})$, $\vdash'=\psi_t\circ\vdash(\psi^{-1}_t\otimes\psi_t^{-1})$, $\varphi'=\psi_t\circ\varphi \circ\psi^{-1}_t$ and 
$\psi'=\psi_t\circ\psi \circ\psi_t^{-1}$. 
If $(A\left[\left[t\right]\right], \dashv_{t}, \vdash_{t} , \varphi, \psi)$ is  BiHom-associative dialgebra then  $(A, \dashv',\vdash', \varphi',\psi')$  
is BiHom-associative dialgebra.
\end{proposition}
\begin{proof}
By straightforward computation, we have
\begin{align*}
(x\dashv'y)\dashv' \psi'(z)
&=\psi_t\circ(\psi^{-1}_t(x)\dashv\psi^{-1}_t(y))\dashv'(\psi_t\circ\psi\circ\psi^{-1}_t(z))\\
&=\psi_t(\psi^{-1}_t\circ\psi_t\circ(\psi^{-1}_t(x)\dashv\psi^{-1}_t(y))\dashv\psi^{-1}_t(\psi_t\circ\psi\circ\psi^{-1}_t(z)))\\
&=\psi_t(\psi^{-1}_t(x)\dashv\psi^{-1}_t(y))\dashv\psi\psi^{-1}(z))\\
&=\psi_t(\varphi\circ\psi^{-1}_t(x)\dashv(\psi^{-1}_t(y)\dashv\psi^{-1}(z)))\\
&=\psi_t(\psi^{-1}_t\circ\psi_t\circ\varphi\circ\psi^{-1}_t(x)\dashv\psi^{-1}_t\circ\psi_t(\psi^{-1}_t(y)\dashv\psi^{-1}_t(z)))\\
&=\psi_t\circ\varphi\circ\psi^{-1}_t(x)\dashv'\psi_t\circ(\psi^{-1}_t(y)\dashv\psi^{-1}_t(z)))=\varphi'(x)\dashv'(y\dashv'z).
\end{align*}
This completes the proof.
\end{proof}
\section{acknowledgment}{We would like to thank the editorial board and referees for their careful study and valuable suggestions within the manuscript.}

\end{document}